\documentclass[11pt,a4paper]{amsart}
\usepackage{amssymb,amsmath,amsthm}
\usepackage{commath}
\usepackage{mathtools}
\usepackage{mathrsfs}
\usepackage{accents}

\usepackage[utf8]{inputenc}

\usepackage{caption}
\usepackage{subcaption}

\usepackage[colorlinks]{hyperref}
\hypersetup{
    allcolors = black
}

\newtheorem{theorem}{Theorem}
\newtheorem{corollary}{Corollary}
\newtheorem{lemma}{Lemma}

\newtheorem{conjecture}{Conjecture}

\newtheorem{proposition}{Proposition}

\newcommand{\R}{\mathbb{R}}
\newcommand{\eps}{\varepsilon}

\newcommand{\la}{\langle}
\newcommand{\ra}{\rangle}

\newcommand{\w}{\widehat{w}}
\newcommand{\inte}{\textrm{int} \,}
\newcommand{\K}{\mathcal{K}}

\DeclareMathOperator{\conv}{conv}

\makeatletter
\@namedef{subjclassname@2020}{%
  \textup{2020} Mathematics Subject Classification}
\makeatother

\author{Gergely Ambrus}

\title{A generalization of Bang's lemma}

\thanks{Research was partially supported by Hungarian National Research grant no. NKFIH KKP-133819 and by the Ministry of Innovation and
Technology of Hungary from the National Research, Development
and Innovation Fund, project no. TKP2021-NVA-09.
}



\date{\today}
\begin{document}

\maketitle

\begin{abstract}
We prove a common extension of Bang's and Kadets' lemmas for contact pairs, in the spirit of the Colourful Carathéodory Theorem. We also formulate a generalized version of the affine plank problem and prove it under special assumptions. In particular, we obtain a generalization of Kadets' theorem.  Finally, we give applications to problems regarding translative and homothetic coverings.
\end{abstract}

\section{Plank problems}
In 1950, Bang~\cite{B50, B51} proved the plank problem of Tarski~\cite{T32}: he showed that if a convex body $K \subset \R^d$ is covered by a finite
number of planks, then the sum of their widths is not less than the minimal
width of $K$. Here a {\em plank}  $P$ is the closed region of $\R^d$ between two parallel hyperplanes, whose distance apart is the {\em width} of $P$, denoted by $w(P)$.  Let $\K^d$ stand for the family of convex bodies in $\R^d$. Given a convex body $K \in \K^d$ and a direction $u \in \R^d \setminus \{0\}$, the {\em width of $K$ in direction $u$}, denoted by $w_u(K)$, is the width of the smallest plank containing $K$ whose bounding hyperplanes are orthogonal to $u$.
The {\em minimal width of $K$} is $w(K) = \min_u w_u(K)$.

In the same article, Bang suggested an affine invariant generalization of the problem. Given a convex body $K \subset \R^d$ and a plank $P \subset \R^d$, he defined the {\em width of $P$ relative to $K$} as
\begin{equation}\label{relwidth}
w_K(P) = \frac{w(P)}{w_u(K)}
\end{equation}
where $u\in\R^d\setminus\{0\}$ is normal to a boundary hyperplane of $P$.
\begin{conjecture}[The affine plank problem, Bang~\cite{B51}]\label{conj_affine}
  Assume that the planks $P_1, \ldots, P_n$ cover the convex body $K\in \K^d$. Then $\sum_{i=1}^n w_K(P_i) \geq 1.$
\end{conjecture}

The statement was proved for symmetric $K$'s by Ball~\cite{B91}, but is still open for general convex bodies apart from the following special cases: only two planks in the plane \cite{B54, M58, A68}, at most three planks in the plane \cite{H93}, or when the planks can be partitioned to two parallel subfamilies~\cite{G88,AKP19}.

One of the main ingredients of Bang's proof of the plank problem is the following statement, which has been polished to its present form by Fenchel~\cite{F51} and Ball~\cite{B01}:
\begin{lemma}[Bang's Lemma]\label{bang}
Let $(u_i)_1^n$ be a sequence of unit vectors in $\R^d$ and $(w_i)_1^n$ a
sequence of positive numbers. Then for any sequence $(m_i)_1^n$ of reals, there
exists a point $u$ of the form
\[
u = \sum \eps_i u_i w_i
\]
with $\eps_i = \pm 1$ for $i \in [n]$, so that
\[
|\la u, u_k \ra - m_k | \geq w_k
\]
holds for each $k$.
\end{lemma}
\noindent
Above and later on, $[n] = \{1, \ldots, n\}.$

Bang's lemma has found numerous applications in the past decades. In particular, it is a crucial ingredient of Ball's proof for the symmetric case of the affine plank problem~\cite{B91}, his lower bound on the density of sphere packings~\cite{B92} as well as Nazarov's solution of the coefficient problem~\cite{N97}.

In 2005, Kadets~\cite{K05} generalized the original plank problem to coverings with arbitrary convex bodies in $\R^d$. He proved that if a family of convex bodies $K_1, \ldots, K_n \in \K^d$ covers $K \in \K^d$, then $\sum_{i=1}^n r(K_i) \geq r(K)$, where $r(K)$ denotes the inradius of $K$. The crux of his argument boils down to the following generalization of Theorem~\ref{bang}. Below, $S^{d-1}$ denotes the $d$-dimensional unit sphere.

\begin{lemma}[Kadets' Lemma]\label{kadets}
Assume that $U_1, \ldots, U_n \subset S^{d-1}$ are finite sets of unit vectors in $\R^d$ so that $0 \in \conv U_i$ for each $i$. Let $r_1, \ldots, r_n >0$ be positive numbers. Then for every set of points $o_1, \ldots, o_n \in \R^d$ there exist $u_i \in U_i$, $i = 1, \ldots, n$ so that setting $u = \sum_{i=1}^n r_i u_i$, \[
\la u - o_k, u_k \ra \geq r_k
\]
holds for every $k$.
\end{lemma}

We note that the planar case of Kadets' theorem was also proved much earlier by Ohmann~\cite{O53}, and later independently by Bezdek~\cite{B07}.
Prior to that, Bezdek and Bezdek \cite{BB95} solved Conway's potato problem and showed that if $K$ is successively sliced by $n-1$ hyperplane cuts, dividing just one piece by each cut, then one of the remaining pieces must have inradius at least $\frac 1 n r(K)$. In a follow-up article \cite{BB96}, they extended their result to {\em $K$-inradii} instead of inradii: given a convex body $K \in \K^d$ and a convex set $L \subset \R^d$, the $K$-inradius of $L$ is defined as
\begin{equation}
r_K(L) = \sup \{ \lambda \geq 0: \ \lambda K + x \subset L \textrm{ for some } x \in \R^d . \}
\end{equation}
Note that for a plank $P \subset \R^d$,
\begin{equation}\label{wrk}
  w_K(P) = r_K(P).
\end{equation}
The connection to plank problems is provided by Alexander~\cite{A68}, who proved that for $K \in \K^d$, the sum of the $K$-inradii of $n$ planks covering $K$ is guaranteed to be at least 1 if and only if for an arbitrary set of $n-1$ hyperplanes, there exists a convex body $L \subset K$ with $r_K(L) \geq \frac 1 n$ not cut by any of these hyperplanes.

Along this direction, Akopyan and Karasev~\cite{AK12} proved analogues of Kadets' result for $K$-inradii: among other results, they showed that if $K_1, \ldots, K_n$ form an {\em inductive partition} of  $K \in \K^d$, then $\sum r_K(K_i) \geq 1$ holds, moreover, the same statement is true in the plane for arbitrary convex partitions. Balitskiy~\cite{B21} generalized the theorems of Bang and Kadets to multi-planks.

The goal of the present paper is to generalize Lemmas~\ref{bang} and \ref{kadets} in the spirit of Bárány's Colourful Carathéodory Theorem~\cite{B82}. The resulting statement may be applied to general covering problems involving $K$-inradii, and in particular, to translative covering problems.

Let $K, L \subset \R^d$ be convex bodies. It is a well-known fact that if $K'$ is a maximal homothetic copy of $K$ inscribed in $L$, then there exists a set of points $u_1, \ldots, u_n \in \R^d$ with corresponding normal directions $v_1, \ldots, v_n \in \R^d\setminus \{ 0 \}$ such that $u_i$ is a common boundary point of $K'$ and $L$ with corresponding (common) outer normal vector $v_i$ for every~$i$, moreover, $0 \in \conv \{ v_1, \ldots, v_n \}$.  The pairs $(u_i, v_i)$ are called {\em contact pairs} of $K'$ and $L$. A set of contact pairs is called {\em complete} if $0 \in \conv \{ v_1, \ldots, v_n \}$. Carathéodory's theorem implies that in the above setting, there always exists a complete set of contact pairs of cardinality at most $d+1$.

We are going to generalize Bang's lemma to sets of  contact pairs. The forthcoming arguments will use the following setup. For vectors $u, v \in \R^d$, we define $w \in \R^d \times \R^d$ as $w = (u,v)$. For any such vector $w =(u,v)$, let $\w = (v,u)$. Here comes the main result of the paper.
\begin{theorem}\label{thm_main}
Assume that $W_1, \ldots, W_n \subset \R^d \times \R^d$ are finite sets such that $(0,0)  \in \conv W_i$ for each $i \in [n]$. For any set of vectors $z_1, \ldots, z_n \in \R^d \times \R^d$, there exist $w_i \in W_i$, $i \in [n]$ so that by setting $w = \sum_{i=1}^n w_i$,
\begin{equation}\label{thm1eq}
\la w - z_k, \w_k \ra \geq \la w_k, \w_k \ra
\end{equation}
holds for each $k$.
\end{theorem}
Theorem~\ref{thm_main} is formulated in the context of contact pairs $(u_i, v_i)$. Setting $v_i = u_i$ and $y_i = x_i$ for every $i$, it takes the following simpler form.
\begin{corollary}\label{cor1}
  Assume that all the finite vector sets $U_1, \ldots, U_n \subset \R^d$ contain the origin in their convex hull. Then for any set of vectors $x_1 , \ldots, x_n \in \R^d$ we may select $u_i \in U_i$ for each $i \in [n]$ so that setting $u = \sum_i u_i$,
  \[
  \la u - x_k, u_k \ra \geq  |u_k|^2
  \]
  holds for every $k$.
\end{corollary}
When all the sets $U_i$ consist of unit vectors, we recover Kadets' lemma, while the case $U_i = \{ -u_i, u_i \}$ with $u_i \in \R^d$ corresponds to Bang's lemma.

Theorem~\ref{thm_main} and Corollary~\ref{cor1} lead towards the following generalization of the affine plank problem. We say that the convex sets $C_1, \ldots, C_n \subset \R^d$ {\em permit a translative covering of $K \in \K^d$} if
\[
K \subset \bigcup_{i=1}^n ( C_i + x_i)
\]
for some $x_1, \ldots, x_n \in \R^d$.

\begin{conjecture}\label{conj1}
  Assume that the convex sets $C_1, \ldots, C_n \subset \R^d$ permit a translative covering of the convex body $B \in \K^d$. Then
  \[
  \sum_{i=1}^n r_B(C_i) \geq 1
  \]
  holds.
\end{conjecture}
Equation~\eqref{wrk} shows that this is indeed an extension of Conjecture~\ref{conj_affine}, the affine plank problem. Balitskiy~\cite{B21} conjectures an even more general statement.

In addition to the special cases of the affine plank problem discussed earlier, Conjecture~\ref{conj1} has been proved if $B$ is an ellipsoid \cite{O53,B07,K05} or  if the sets $C_i$ form a partition of $B$ in the plane, or an inductive partition in higher dimensions \cite{AK12}. Corollary~\ref{cor1} implies that it also holds in a wide range of cases.

\begin{theorem}\label{thm_symm}
Conjecture~\ref{conj1} holds if for every $i \in [n]$ there exists some $o_i \in \R^d$ such that $r_B(C_i) B-o_i$ and $C_i-o_i$ have a complete set of contact pairs $W_i \subset \R^d \times \R^d$ with $(0,0) \in \conv W_i$, so that for any two such contact pairs $(u_i, v_i) \in W_i$, $(u_j , v_j) \in W_j$ with $i \neq j$,
\begin{equation}\label{cond_symm}
\la u_i, v_j \ra = \la u_j, v_i \ra
\end{equation}
holds.
\end{theorem}

We immediately obtain the following generalization of Kadets' theorem~\cite{K05}.

\begin{corollary}\label{cor_identical}
  Conjecture~\ref{conj1} holds if for every $i\in[n]$ there exists some $o_i \in \R^d$ such that $r_B(C_i) B-o_i$ and $C_i-o_i$ have a complete set of contact pairs of the form $(u,u)$.
\end{corollary}

A particular case is when $o_i \in r_B(C_i)B$, and the contact points between $r_B(C_i) B$ and $C_i$ are the local extrema of the radial function $|x - o_i|$ for $x \in \partial (r_B(C_i) B)$, provided that $0$ is contained in the convex hull of these. Such a situation is illustrated on Figure~\ref{fig1}.

\begin{figure}[h]
  \centering
  \includegraphics[width = 0.5 \textwidth]{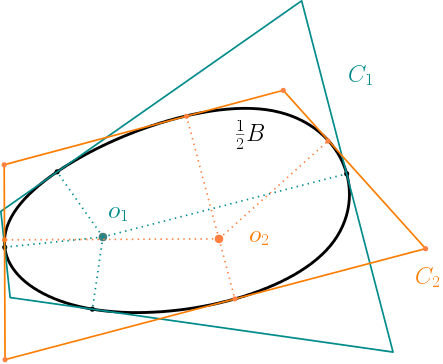}
  \caption{Convex discs with complete sets of contact pairs of the form $(u,u)$}
\label{fig1}
\end{figure}

The direct application of Theorem~\ref{thm_main} yields another sufficient condition.

\begin{proposition}\label{prop2}
Conjecture~\ref{conj1} holds in $\R^{2d}$ if for every $i$ there exists some $o_i \in \R^{2d}$ such that $r_B(C_i) B-x_i$ and $C_i-o_i$ have a complete set of contact pairs of the form $(w,\widehat{w})$.
\end{proposition}

Applications of Theorem~\ref{thm_main} to translative coverings are listed in Section~\ref{section_covering}.

\medskip

Although the above results are formulated for finite vector sets/families of convex sets in $\R^d$, they may be extended to an arbitrary number of vectors/convex sets in finite dimensional real or complex Hilbert spaces using the standard techniques.

\medskip

Further developments, historical and mathematical details  related to the plank problem may be found in \cite{A10, B14, FT22+}.

\section{Proof of the main results}

\begin{proof}[Proof of Theorem~\ref{thm_main}]
For each $i$, let $w_i =(u_i, v_i)$ and $z_i = (x_i, y_i)$ with $u_i , v_i, x_i, y_i \in \R^d$.
Select $w_i \in W_i$, $i\in [n]$ so as to maximize
\begin{equation}\label{quant}
\sum_{i \neq j} \la u_i, v_j \ra - \sum_i \la x_i, v_i \ra - \sum_j \la u_j, y_j \ra
\end{equation}
and set $w =  (u,v) = \sum_i w_i $, that is, $u = \sum_i u_i$ and $v = \sum v_i$. We will show that \eqref{thm1eq} holds for every $k$, that is,
\begin{equation}\label{ukvkeq}
\la u - x_k, v_k \ra + \la u_k , v - y_k \ra \geq 2 \la u_k, v_k \ra.
\end{equation}

Let $k\in[n]$ be arbitrary. By the condition of the theorem, there exist non-negative numbers $\alpha(w_k'),\ w_k' \in W_k$ so that $\sum_{w_k' \in W_k}\alpha(w_k') = 1$ and
\[
\sum_{w_k' \in W_k}\alpha(w_k') w_k'= (0,0).
\]
Moreover, since \eqref{quant} is maximal, for each $w_k' = (u_k',v_k') \in W_k$,
\[
0 \geq \sum_{i \neq k} \la u_i, v_k' - v_k \ra + \sum_{j \neq k} \la u_k' - u_k, v_j \ra - \la x_k, v_k' - v_k \ra - \la u_k' - u_k, y_k \ra.
\]
Multiplying the above equation by $\alpha(w_k')$ and summing up for all $w_k' \in W_k$ leads to
\[
0 \geq \sum_{i \neq k} \la u_i, - v_k \ra + \sum_{j \neq k} \la - u_k, v_j \ra - \la x_k,- v_k  \ra - \la - u_k, y_k \ra,
\]
which directly implies \eqref{ukvkeq}.
\end{proof}

\begin{proof}[Proof of Theorem~\ref{thm_symm}]
We may assume that $0 \in B$.
Let $\lambda_i = r_{B}(C_i)$ for every $i$, and $\lambda := \sum \lambda_i$.  Assume on the contrary that $\lambda  < 1$ and $B \subset \bigcup (C_i + x'_i)$ with some $x'_1, \ldots, x'_n \in \R^d$. Choose $\eps>0$ so that $(1 + \eps) \lambda <1$. For each $i$, let $W_i$ be the complete set of contact pairs between $\lambda_i B - o_i$ and $C_i - o_i$ which contains $(0,0)$ in its convex hull. Also, set $ o =(1 + \eps) \sum o_i$ and $x_i = x'_i + o_i - o$ for each $i$.

Apply Theorem~\ref{thm_main} to the sets $(1 + \eps) W_i$ and the corresponding  points $z_k = (2 x_i,0)$. It implies the existence of $w_i = (u_i, v_i)\in W_i$, $i \in [n]$ so that setting $w=(u,v)= \sum (1 + \eps) w_i$,
\[
\la u - 2 x_k , (1 + \eps) v_k \ra + \la v, (1 + \eps) u_k \ra \geq 2 (1 + \eps)^2 \la u_k, v_k \ra
\]
holds for each $k$. Since \eqref{cond_symm} implies that $\la u, v_k \ra = \la v, u_k \ra$, the above equation simplifies to
\[
\la u - x_k, v_k \ra \geq (1 + \eps) \la u_k, v_k \ra >\la u_k, v_k \ra.
\]
 Since $u_k$ is a boundary point of $C_k - o_k$ with outer normal $v_k$, the convexity of $C_k$ implies that $u - x_k \not \in C_k - o_k$, equivalently, $u+o \not \in C_k + x_k'$ for any $k$. On the other hand, $u_i \in \lambda_i B - o_i$ for every $i$. Therefore,
\[
u \in \sum_i (1 + \eps)  \lambda_i B - \sum_i (1 + \eps) o_i = (1 + \eps)\lambda B - o.
\]
Since $B$ is  convex and $0 \in B$, $(1 + \eps)\lambda B \subset B$. Hence, $u + o \in B$, but it is not covered by any of the sets $C_k + x_k'$, which is a contradiction.
\end{proof}

The proof of Proposition~\ref{prop2} is nearly identical, thus we leave it to the dedicated reader.

%


\section{Applications to translative coverings}\label{section_covering}

Corollary~\ref{cor1} readily implies the next statement regarding translative coverings.

\begin{proposition}\label{prop_covering}
  Let $\K = \{ K_1, \ldots, K_n \}$ be a family of convex bodies in $\R^d$ containing the origin in their interior. For each $i$, let $V_i \subset S^{d-1}$ be a set of direction vectors for which $0 \in \conv V_i$.  Denote by $U_i$ the set of projection vectors of $0$ onto the supporting hyperplanes of $K_i$ corresponding to members of $V_i$. Then $\inte \K =\{\inte K_1, \ldots, \inte K_n \} $ does not permit a translative covering of $U_1 + \ldots + U_n$.
\end{proposition}

A particular case is when all the $K_i$'s are homothets of a fixed convex body $K \in \K^d$. Such homothetic coverings have been studied extensively, see e.g. \cite{N18}, Section 3.2 of \cite{BMP05} and Section 15.4 of \cite{FT22+}. A related conjecture is due to Soltan~\cite{S90}:

\begin{conjecture}[V. Soltan]\label{conj_soltan}
  Assume that $K \in \K^d$ and that $\lambda_1 K, \ldots, \lambda_n K$ permit a translative covering of $K$ with $\lambda_i \in (0,1)$ for every $i$. Then
  \[
  \sum_{i=1}^n \lambda_i \geq d.
  \]
\end{conjecture}
Let $T^d$ denote the $d$-dimensional regular simplex. Setting $K=T^d$, $n= d+1$ and $\lambda_i = \frac {d}{d+1}$ shows that the above bound may not be improved.

Conjecture~\ref{conj_soltan} was proved for $d=2$ or $n=d+1$ by Soltan and Vásárhelyi~\cite{SV93} and for $K= B^d$ by Glazyrin~\cite{G19}, while Naszódi~\cite{N10} showed that $\sum \lambda_i > \alpha \, d$ for any fixed $\alpha<1$ if $d$ is sufficiently large.

Vásárhelyi~\cite{V84} showed that given a triangle $T$ in the plane of area 1, any family of its negative homothets whose areas sum up to at least 4 permits a translative covering of $T$. We prove a reverse statement resembling Conjecture~\ref{conj_soltan} which holds in arbitrary dimensions.

\begin{theorem}\label{thm_simplex}
Assume that $T\subset \R^d$ is a non-degenerate simplex, and $\lambda_1, \ldots, \lambda_n \geq 0$ are so that the family $-\lambda_1 T, \ldots, -\lambda_n T$ permits a translative covering of $T$. Then
\begin{equation}\label{eq_simplex}
  \sum_{i=1}^n \lambda_i \geq d.
\end{equation}
\end{theorem}

\begin{proof}
We may assume that $T = T^d$ with its centre at  $0$. Let $V$ be the set of normal directions of the facets of $T^d$ and $U$ be the set of projection vectors of $0$ onto the facets. It is well-known that $\conv U = - \frac 1 d T^d$. Let $\eps>0$ be arbitrary. Applying Proposition~\ref{prop_covering} with $K_i = - (1+ \eps) \lambda_i T^d$, $V_i = V$ and $U_i = - (1+ \eps)\lambda_i U$  yields an uncovered point in
\[
U_1 + \ldots + U_n \subset (1 + \eps) \frac {\lambda_1 + \ldots + \lambda_n}{d} \, T^d,
\]
which shows that  \eqref{eq_simplex} must hold true.
\end{proof}

Based on Theorem~\ref{thm_simplex}, G. Fejes Tóth~\cite{FT22+} formulated the following extension of Soltan's conjecture to coefficients of arbirary sign.
\begin{conjecture}[G. Fejes Tóth]\label{conj_soltan}
  Assume that $K \in \K^d$ and that $\lambda_1 K, \ldots, \lambda_n K$ permit a translative covering of $K$ with $\lambda_i \in (-1,1)$ for every $i$. Then
  \[
  \sum_{i=1}^n |\lambda_i| \geq d.
  \]
\end{conjecture}

\section{Acknowledgements}

I am grateful to G. Fejes Tóth, O. Ortega Moreno, and J. Pach  for the illuminating conversations on the subject.

\medskip

I would like to dedicate this piece of work to the loving memory of my father, Imre Ambrus (1953-2021.)

\vspace{1 cm}
\noindent
{\sc Gergely Ambrus}
\smallskip

\noindent
{\em Alfréd Rényi Institute of Mathematics, Eötvös Loránd Research Network, Budapest, Hungary, and Bolyai Institute, University of Szeged, Hungary}

\noindent
e-mail address: \texttt{ambrus@renyi.hu}

\end{document}